\documentclass[12pt,letterpaper,final,twoside,leqno]{amsart}
\usepackage{amsmath,amsthm,amssymb,amsfonts,amscd,amsopn} %
\usepackage{eucal,mathrsfs,relsize}
\usepackage{xspace,chngcntr}
\usepackage{rotating,mathtools}
\usepackage[all,cmtip]{xy}\xyoption{dvips}
\usepackage{tikz-cd}
\usetikzlibrary{arrows.meta}
 
\usepackage{comment} 
\usepackage{supertabular}
\usepackage
{hyperref}
\hypersetup{
    colorlinks=true,
    linkcolor={blue!50!black},
    citecolor={green!50!black},
    urlcolor={red!80!black}
}

\usepackage{datetime,chngcntr}
\usepackage{paralist}

\usepackage{enumitem}
\usepackage{eurosym}

\usepackage{mfirstuc}
\DeclareMathAlphabet{\smallchanc}{OT1}{pzc}%
                                 {m}{it}
\DeclareFontFamily{OT1}{pzc}{}
\DeclareFontShape{OT1}{pzc}{m}{it}%
             {<-> s * [1.100] pzcmi7t}{}
\DeclareMathAlphabet{\mathchanc}{OT1}{pzc}%
                                 {m}{it}

   [2017/06/13 v0.1 finetuning mabliautoref and theorem setup]

\RequirePackage{hyperref}

\newcommand{\sectionsize}{\relsize{-.5}}
\newcommand{\theoremsize}{\relsize{-.5}}
\newcommand{\mrsize}{\relsize{-.8}}

\renewcommand{\subsectionautorefname}{\sectionsize\sf \subsectionautorefname}

\makeatletter
\@ifdefinable\equationname{\let\equationname\equationautorefname}
\def\equationautorefname~#1\@empty\@empty\null{\protect{\theoremsize\sf
    (#1\@empty\@empty\null)}}%
\@ifdefinable\AMSname{\let\AMSname\AMSautorefname}
\def\AMSautorefname~#1\@empty\@empty\null{\sf( #1\@empty\@empty\null)}%
\@ifdefinable\itemname{\let\itemname\itemautorefname}
\def\itemautorefname~#1\@empty\@empty\null{\mrsize{%
    {\sf #1}}\@empty\@empty\null%
}%
\makeatother

\RequirePackage{amsthm}
\RequirePackage{aliascnt}
\newcommand{\basetheorem}[3]{%
    \newtheorem{#1}{#2}[#3]
    \newtheorem*{#1*}{#2}
    \expandafter\def\csname #1autorefname\endcsname{#2}
}%
\newcommand{\maketheorem}[3]{%
    \newaliascnt{#1}{#2}
    \newtheorem{#1}[#1]{\theoremsize\sf #3}
    \aliascntresetthe{#1}
    \expandafter\def\csname #1autorefname\endcsname{\theoremsize\sf #3}
    \newtheorem{#1*}{#3}
}%
\newcommand{\baseremark}[3]{%
    \newtheorem{#1}{#2}{#3}
    \newtheorem*{#1*}{#2}
    \expandafter\def\csname #1autorefname\endcsname{#2}
}%
\newcommand{\makeremark}[3]{%
    \newaliascnt{#1}{#2}
    \newtheorem{#1}[equation]{#3}
    \aliascntresetthe{#1}
    \expandafter\def\csname #1autorefname\endcsname{\theoremsize\sf #3}
    \newtheorem{#1*}{#3}
}%
\theoremstyle{plain}   

\basetheorem{theorem}{Theorem}{section}

\newcommand{\mypagesize}{
\textwidth= 6.5in
\textheight=8.75in
\voffset-.5in
\hoffset-.75in
\marginparwidth=56pt
\footskip.5in
}

\newcounter{are-there-sections}
\mypagesize
\usepackage{amsbsy}

\usepackage{marvosym}

\DeclareMathAlphabet{\smallchanc}{OT1}{pzc}%
                                 {m}{it}
\DeclareFontFamily{OT1}{pzc}{}
\DeclareFontShape{OT1}{pzc}{m}{it}%
             {<-> s * [1.100] pzcmi7t}{}
\DeclareMathAlphabet{\mathchanc}{OT1}{pzc}%
                                 {m}{it}

\newcommand{\mcL}{\mathchanc{L}}

\newcommand{\mcR}{\mathchanc{R}}

\newcommand{\sA}{\mathscr{A}}

\newcommand{\sO}{\mathscr{O}}

\newcommand{\sfw}{{\sf w}}

\newcommand{\bN}{\mathbb{N}}

\newcommand{\bZ}{\mathbb{Z}}

\DeclareSymbolFont{largesymbolsA}{U}{jkpexa}{m}{n}
\SetSymbolFont{largesymbolsA}{bold}{U}{jkpexa}{bx}{n}
\DeclareMathSymbol{\varprod}{\mathop}{largesymbolsA}{16}
\makeatletter
\newcommand{\LeftEqNo}{\let\veqno\@@leqno}
\makeatother
\newcommand{\ul}{\underline}

\newcommand{\into}{\hookrightarrow}

\newcommand{\properideal}%
        {\subsetneq}

\newcommand{\leteq}{\colon\!\!\!=}

\newcommand{\coh}{\mathrm{coh}}

\DeclareMathOperator{\cone}{{Cone}}
\DeclareMathOperator{\depth}{{depth}}

\newcommand{\filt}{{\operatorname{filt}}}

\DeclareMathOperator{\Hom}{Hom}

\DeclareMathOperator{\id}{{id}}

\newcommand{\hotimes}[0]{%
  \protect{\ensuremath{\kern.1em{{%
  \raisebox{.5\depth}{$\scriptstyle[$}
  }}\kern-.25em\otimes\kern-.25em{%
  \raisebox{.5\depth}{$\scriptstyle]$}
  }\kern.1em}}}

\DeclareMathOperator{\obj}{{Obj}}

\newcommand{\factor}[2]{\left. \raise .2em\hbox{\ensuremath{#1}\vphantom{$I^d$}}
\hskip -.1em \right/ \hskip -.4em \raise -.3em\hbox{\ensuremath{#2}}}%
\newcommand\mtimes[3]{{\varprod_{#1}^{#2}}_{\raise 1ex \hbox{\scriptsize #3}}}%

\newcommand{\myR}{{\mcR\!}}
\newcommand{\myL}{{\mcL\!}}

\newcommand{\kdot}{{{\,\begin{picture}(1,1)(-1,-2)\circle*{2}\end{picture}\,}}}

\newcommand{\cmx}[1]{{#1}^{\raisebox{.15em}{\ensuremath\kdot}}}

\newcommand{\Om}{\underline{\Omega}}

\newcommand{\hypres}[1]{{\varepsilon_\kdot:{#1}_\kdot \to {#1}}}

\def\dimcoh#1.#2.#3.{h^{#1}(#2,#3)}
\def\hypcoh#1.#2.#3.{\mathbb H_{\vphantom{l}}^{#1}(#2,#3)}
\def\loccoh#1.#2.#3.#4.{H^{#1}_{#2}(#3,#4)}
\def\dimloccoh#1.#2.#3.#4.{h^{#1}_{#2}(#3,#4)}
\def\lochypcoh#1.#2.#3.#4.{\mathbb H^{#1}_{#2}(#3,#4)}
\def\seslong#1.#2.#3.{0  \longrightarrow  #1   \longrightarrow 
 #2 \longrightarrow #3 \longrightarrow 0} 
\def\sesshort#1.#2.#3.{0
 \rightarrow #1 \rightarrow #2 \rightarrow #3 \rightarrow 0}
\def\dist#1.#2.#3.{  #1   \longrightarrow 
 #2 \longrightarrow #3 \stackrel{+1}{\longrightarrow} } 
\def\CDdist#1.#2.#3.{  #1   @>>>  #2  @>>>   #3 @>+1>> }  
\def\shortses#1.#2.#3.{0  \rightarrow  #1   \rightarrow 
 #2  \rightarrow   #3 \rightarrow  0}
\def\shortdist#1.#2.#3.{  #1   \rightarrow 
 #2  \rightarrow   #3 \stackrel{+1}{\rightarrow} }  
\def\ddist#1.#2.#3.#4.#5.#6.{\CD
#1 @>>> #2 @>>> #3 @>+1>> \\
@VVV @VVV @VVV \\
#4 @>>> #5 @>>> #6 @>+1>> 
\endCD}
\def\ddistun#1.#2.#3.#4.#5.#6.{\CD
#1 @>>> #2 @>>> #3 @>+1>> \\
@. @VVV @VVV  \\
#4 @>>> #5 @>>> #6 @>+1>> 
\endCD}
\def\Iff#1#2#3{
\hfil\hbox{\hsize =#1
\vtop{\noin #2}
\hskip.5cm 
\lower.5\baselineskip\hbox{$\Leftrightarrow$}\hskip.5cm
\vtop{\noin #3}}\hfil\medskip}
\newcommand{\union}\cup
\newcommand{\intersect}\cap
\newcommand{\Union}\bigcup
\newcommand{\Intersect}\bigcap
\def\myoplus#1.#2.{\underset #1 \to {\overset #2 \to \oplus}}

\newcommand{\resto}[1]{\raise -.5ex\hbox{$\vert$}_{#1}}

\def\qis{\,{\simeq}_{\text{qis}}\,}

\newcommand{\ses}{short exact sequence\xspace}
\newcommand{\sess}{short exact sequences\xspace}

\newcommand{\DB}{Du~Bois\xspace}

\begin{document}
\makeatletter
\definecolor{brick}{RGB}{204,0,0}
\def\@cite#1#2{{%
 \m@th\upshape\mdseries[{\sffamily\relsize{-.5}#1}{\if@tempswa,
   \sffamily\relsize{-.5}\color{brick} #2\fi}]}}
\newcommand{\sandor}{{\color{blue}{S\'andor \mdyydate\today}}}
\newenvironment{refmr}{}{}
%
\definecolor{refblue}{RGB}{0,0,128}
        \newcommand\refblue{\color{refblue}}
\newcommand\james{M\hskip-.1ex\raise .575ex \hbox{\text{c}}\hskip-.075ex Kernan\xspace}
\newcommand\ifft{if and only if\xspace}
\newcommand\sfref[1]{{\sf\protect{\relsize{-.5}\ref{#1}}}}
\newcommand\stepref[1]{{\sf\protect{\relsize{-.5}\refblue Step~\ref{#1}}}}
\newcommand\demoref[1]{{\sf(\protect{\relsize{-.5}\ref{#1}})}}
\renewcommand\eqref{\demoref}

%
\renewcommand\thesubsection{\thesection.\Alph{subsection}}
\renewcommand\subsection{
  \renewcommand{\sfdefault}{phv}
  \@startsection{subsection}%
  {2}{0pt}{-\baselineskip}{.2\baselineskip}{\raggedright
    \sffamily\itshape\relsize{-.5}
  }}
\renewcommand\section{
  \renewcommand{\sfdefault}{phv}
  \@startsection{section} %
  {1}{0pt}{\baselineskip}{.2\baselineskip}{\centering
    \sffamily
    \scshape
}}

\setlist[enumerate, 1]{itemsep=3pt,topsep=3pt,leftmargin=1.5em,font=\upshape,
  label={(\roman*)}}
\newlist{enumfull}{enumerate}{1}
\setlist[enumfull]{itemsep=3pt,topsep=3pt,leftmargin=3.25em,font=\upshape,
  label={(\thethm.\arabic*\/)}}
\newlist{enumbold}{enumerate}{1}
\setlist[enumbold]{itemsep=3pt,topsep=3pt,leftmargin=1.95em,font=\upshape,
  label={
    \textbf{({\it\textbf{\roman*}}\/)}}}
\newlist{enumalpha}{enumerate}{1}
\setlist[enumalpha]{itemsep=3pt,topsep=3pt,leftmargin=2em,font=\upshape,
  label=(\alph*\/)}
\newlist{widemize}{itemize}{1}
\setlist[widemize]{itemsep=3pt,topsep=3pt,leftmargin=1em,label=$\bullet$}
\newlist{widenumerate}{enumerate}{1}
\setlist[widenumerate]{itemsep=3pt,topsep=3pt,leftmargin=1.75em,label=(\roman*)}
\newlist{widenumalpha}{enumerate}{1}
\setlist[widenumalpha]{itemsep=3pt,topsep=3pt,leftmargin=1.5em,label=(\alph*)}
\newcounter{parentthmnumber}
\setcounter{parentthmnumber}{0}
\newcounter{currentparentthmnumber}
\setcounter{currentparentthmnumber}{0}
\newcounter{nexttag}
\newcommand{\setnexttag}{%
  \setcounter{nexttag}{\value{enumi}}%
  \addtocounter{nexttag}{1}%
}
\newcommand{\placenexttag}{%
\tag{\roman{nexttag}}%
}

\newenvironment{thmlista}{%
\label{parentthma}
\begin{enumerate}
}{%
\end{enumerate}
}
\newlist{thmlistaa}{enumerate}{1}
\setlist[thmlistaa]{label=(\arabic*), ref=\autoref{parentthm}\thethm(\arabic*)}
\newcommand*{\parentthmlabeldef}{%
  \expandafter\newcommand
  \csname parentthm\the\value{parentthmnumber}\endcsname
}
\newcommand*{\ptlget}[1]{%
  \romannumeral-`\x
  \ltx@ifundefined{parentthm\number#1}{%
    \ltx@space
    \parentthmundefined
  }{%
    \expandafter\ltx@space
    \csname mymacro\number#1\endcsname
  }%
}  
\newcommand*{\parentthmundefined}{\textbf{??}}
\parentthmlabeldef{parentthm}
\newenvironment{thmlistr}{%
\label{parentthm}
\begin{thmlistrr}}{%
\end{thmlistrr}}
\newlist{thmlistrr}{enumerate}{1}
\setlist[thmlistrr]{label=(\roman*), ref=\autoref{parentthm}(\roman*)}
\newcounter{proofstep}%
\setcounter{proofstep}{0}%
\newcommand{\pstep}[1]{%
  \smallskip
  \noindent
  \emph{{\sc Step \arabic{proofstep}:} #1.}\addtocounter{proofstep}{1}}
\newcounter{lastyear}\setcounter{lastyear}{\the\year}
\addtocounter{lastyear}{-1}
\newcommand\sideremark[1]{%
\normalmarginpar
\marginpar
[
\hskip .45in
\begin{minipage}{.75in}
\tiny #1
\end{minipage}
]
{
\hskip -.075in
\begin{minipage}{.75in}
\tiny #1
\end{minipage}
}}
\newcommand\rsideremark[1]{
\reversemarginpar
\marginpar
[
\hskip .45in
\begin{minipage}{.75in}
\tiny #1
\end{minipage}
]
{
\hskip -.075in
\begin{minipage}{.75in}
\tiny #1
\end{minipage}
}}
\newcommand\Index[1]{{#1}\index{#1}}
\newcommand\inddef[1]{\emph{#1}\index{#1}}
\newcommand\noin{\noindent}
\newcommand\hugeskip{\bigskip\bigskip\bigskip}
\newcommand\smc{\sc}
\newcommand\dsize{\displaystyle}
\newcommand\sh{\subheading}
\newcommand\nl{\newline}
\newcommand\toappear{\rm (to appear)}
\newcommand\mycite[1]{[#1]}
\newcommand\myref[1]{(\ref{#1})}
\newcommand{\parref}[1]{\eqref{\bf #1}}
\newcommand\myli{\hfill\newline\smallskip\noindent{$\bullet$}\quad}
\newcommand\vol[1]{{\bf #1}\ } 
\newcommand\yr[1]{\rm (#1)\ } 
\newcommand\cf{cf.\ \cite}
\newcommand\mycf{cf.\ \mycite}
\newcommand\te{there exist\xspace}
\newcommand\st{such that\xspace}
\newcommand\CM{Cohen-Macaulay\xspace}
\newcommand\GR{Grauert-Riemenschneider\xspace}
\newcommand\notinclass{{\relsize{-.5}\sf [not discussed in class]}\xspace}
\newcommand\myskip{3pt}
\newtheoremstyle{bozont}{3pt}{3pt}%
     {\itshape}
     {}
     {\bfseries}
     {.}
     {.5em}
     {\thmname{#1}\thmnumber{ #2}\thmnote{\normalsize
        \ \rm #3}}
\newtheoremstyle{bozont-sub}{3pt}{3pt}%
     {\itshape}
     {}
     {\bfseries}
     {.}
     {.5em}
     {\thmname{#1}\ \arabic{section}.\arabic{thm}.\thmnumber{#2}\thmnote{\normalsize
 \ \rm #3}}
\newtheoremstyle{bozont-named-thm}{3pt}{3pt}%
     {\itshape}
     {}
     {\bfseries}
     {.}
     {.5em}
     {\thmname{#1}\thmnumber{#2}\thmnote{ #3}}
\newtheoremstyle{bozont-named-bf}{3pt}{3pt}%
     {}
     {}
     {\bfseries}
     {.}
     {.5em}
     {\thmname{#1}\thmnumber{#2}\thmnote{ #3}}
\newtheoremstyle{bozont-named-sf}{3pt}{3pt}%
     {}
     {}
     {\sffamily}
     {.}
     {.5em}
     {\thmname{#1}\thmnumber{#2}\thmnote{ #3}}
\newtheoremstyle{bozont-named-sc}{3pt}{3pt}%
     {}
     {}
     {\scshape}
     {.}
     {.5em}
     {\thmname{#1}\thmnumber{#2}\thmnote{ #3}}
\newtheoremstyle{bozont-named-it}{3pt}{3pt}%
     {}
     {}
     {\itshape}
     {.}
     {.5em}
     {\thmname{#1}\thmnumber{#2}\thmnote{ #3}}
\newtheoremstyle{bozont-sf}{3pt}{3pt}%
     {}
     {}
     {\sffamily}
     {.}
     {.5em}
     {\thmname{#1}\thmnumber{ #2}\thmnote{\normalsize
 \ \rm #3}}
\newtheoremstyle{bozont-sc}{3pt}{3pt}%
     {}
     {}
     {\scshape}
     {.}
     {.5em}
     {\thmname{#1}\thmnumber{ #2}\thmnote{\normalsize
 \ \rm #3}}
\newtheoremstyle{bozont-remark}{3pt}{3pt}%
     {}
     {}
     {\scshape}
     {.}
     {.5em}
     {\thmname{#1}\thmnumber{ #2}\thmnote{\normalsize
 \ \rm #3}}
\newtheoremstyle{bozont-subremark}{3pt}{3pt}%
     {}
     {}
     {\scshape}
     {.}
     {.5em}
     {\thmname{#1}\ \arabic{section}.\arabic{thm}.\thmnumber{#2}\thmnote{\normalsize
 \ \rm #3}}
\newtheoremstyle{bozont-def}{3pt}{3pt}%
     {}
     {}
     {\bfseries}
     {.}
     {.5em}
     {\thmname{#1}\thmnumber{ #2}\thmnote{\normalsize
 \ \rm #3}}
\newtheoremstyle{bozont-reverse}{3pt}{3pt}%
     {\itshape}
     {}
     {\bfseries}
     {.}
     {.5em}
     {\thmnumber{#2.}\thmname{ #1}\thmnote{\normalsize
 \ \rm #3}}
\newtheoremstyle{bozont-reverse-sc}{3pt}{3pt}%
     {\itshape}
     {}
     {\scshape}
     {.}
     {.5em}
     {\thmnumber{#2.}\thmname{ #1}\thmnote{\normalsize
 \ \rm #3}}
\newtheoremstyle{bozont-reverse-sf}{3pt}{3pt}%
     {\itshape}
     {}
     {\sffamily}
     {.}
     {.5em}
     {\thmnumber{#2.}\thmname{ #1}\thmnote{\normalsize
 \ \rm #3}}
\newtheoremstyle{bozont-remark-reverse}{3pt}{3pt}%
     {}
     {}
     {\sc}
     {.}
     {.5em}
     {\thmnumber{#2.}\thmname{ #1}\thmnote{\normalsize
 \ \rm #3}}
\newtheoremstyle{bozont-def-reverse}{3pt}{3pt}%
     {}
     {}
     {\sffamily}
     {.}
     {.5em}
     {\thmnumber{{\relsize{-.5}(#2)}}\thmname{ {\relsize{-.25}#1}}\thmnote{{\relsize{-.5}\ \sf #3}}}
\newtheoremstyle{bozont-def-newnum-reverse}{3pt}{3pt}%
     {}
     {}
     {\bfseries}
     {}
     {.5em}
     {\thmnumber{#2.}\thmname{ #1}\thmnote{\normalsize
 \ \rm #3}}
\newtheoremstyle{bozont-def-newnum-reverse-plain}{3pt}{3pt}%
   {}
   {}
   {}
   {}
   {.5em}
   {\thmnumber{\!(#2)}\thmname{ #1}\thmnote{\normalsize
 \ \rm #3}}
\newtheoremstyle{bozont-number}{3pt}{3pt}%
   {}
   {}
   {}
   {}
   {0pt}
   {\thmnumber{\!(#2)} }
\newtheoremstyle{bozont-say}{3pt}{3pt}%
     {}
     {}
     {\sffamily}
     {.}
     {.5em}
     {\thmnumber{{\relsize{-.5}\S#2}}\thmname{ {\relsize{-.25}#1}}%
       \ \thmnote{
         \it #3}}
\newtheoremstyle{bozont-subsay}{3pt}{3pt}%
     {}
     {}
     {\sffamily}
     {.}
     {.5em}
     {\thmnumber{{\relsize{-.5}\S\S#2}}\thmname{ {\relsize{-.25}#1}}\thmnote{{\relsize{-.5}\ \sf #3}}}
\newtheoremstyle{bozont-step}{3pt}{3pt}%
   {\itshape}
   {}
   {\scshape}
   {}
   {.5em}
   {$\boxed{\text{\sc \thmname{#1}~\thmnumber{#2}:\!}}$}
\theoremstyle{bozont}    
\ifnum \value{are-there-sections}=0 {%
  \basetheorem{proclaim}{Theorem}{}
} 
\else {%
  \basetheorem{proclaim}{Theorem}{section}
} 
\fi
\maketheorem{thm}{proclaim}{Theorem}
\maketheorem{mainthm}{proclaim}{Main Theorem}
\maketheorem{cor}{proclaim}{Corollary} 
\maketheorem{cors}{proclaim}{Corollaries} 
\maketheorem{lem}{proclaim}{Lemma} 
\maketheorem{prop}{proclaim}{Proposition} 
\maketheorem{conj}{proclaim}{Conjecture}
\basetheorem{subproclaim}{Theorem}{proclaim}
\maketheorem{sublemma}{subproclaim}{Lemma}
\newenvironment{sublem}{%
\setcounter{sublemma}{\value{equation}}
\begin{sublemma}}
{\end{sublemma}}
\theoremstyle{bozont-sub}
\maketheorem{subthm}{equation}{Theorem}
\maketheorem{subcor}{equation}{Corollary} 
\maketheorem{subprop}{equation}{Proposition} 
\maketheorem{subconj}{equation}{Conjecture}
\theoremstyle{bozont-named-thm}
\maketheorem{namedthm}{proclaim}{}
\theoremstyle{bozont-sc}
\newtheorem{proclaim-special}[proclaim]{\specialthmname}
\newenvironment{proclaimspecial}[1]
     {\def\specialthmname{#1}\begin{proclaim-special}}
     {\end{proclaim-special}}
\theoremstyle{bozont-subremark}
\basetheorem{subremark}{Remark}{proclaim}
\maketheorem{subrem}{equation}{Remark}
\maketheorem{subnotation}{equation}{Notation} 
\maketheorem{subassume}{equation}{Assumptions} 
\maketheorem{subobs}{equation}{Observation} 
\maketheorem{subexample}{equation}{Example} 
\maketheorem{subex}{equation}{Exercise} 
\maketheorem{inclaim}{equation}{Claim} 
\maketheorem{subquestion}{equation}{Question}
\theoremstyle{bozont-remark}
\basetheorem{remark}{Remark}{proclaim}
\makeremark{subclaim}{subremark}{Claim}
\maketheorem{rem}{proclaim}{Remark}
\maketheorem{claim}{proclaim}{Claim} 
\maketheorem{notation}{proclaim}{Notation} 
\maketheorem{assume}{proclaim}{Assumptions} 
\maketheorem{assumeone}{proclaim}{Assumption} 
\maketheorem{obs}{proclaim}{Observation} 
\maketheorem{example}{proclaim}{Example} 
\maketheorem{examples}{proclaim}{Examples} 
\maketheorem{complem}{equation}{Complement}
\maketheorem{const}{proclaim}{Construction}   
\maketheorem{ex}{proclaim}{Exercise} 
\newtheorem{case}{Case} 
\newtheorem{subcase}{Subcase}   
\newtheorem{step}{Step}
\newtheorem{approach}{Approach}
\maketheorem{Fact}{proclaim}{Fact}
\newtheorem{fact}{Fact}
\newtheorem*{SubHeading*}{\SubHeadingName}%
\newtheorem{SubHeading}[proclaim]{\SubHeadingName}
\newtheorem{sSubHeading}[equation]{\sSubHeadingName}
\newenvironment{demo}[1] {\def\SubHeadingName{#1}\begin{SubHeading}}
  {\end{SubHeading}}%
\newenvironment{subdemo}[1]{\def\sSubHeadingName{#1}\begin{sSubHeading}}
  {\end{sSubHeading}} %
\newenvironment{demor}[1]{\def\SubHeadingName{#1}\begin{SubHeading-r}}
  {\end{SubHeading-r}}%
\newenvironment{subdemor}[1]{\def\sSubHeadingName{#1}\begin{sSubHeading-r}}
  {\end{sSubHeading-r}} %
\newenvironment{demo-r}[1]{%
  \def\SubHeadingName{#1}\begin{SubHeading-r}}
  {\end{SubHeading-r}}%
\newenvironment{subdemo-r}[1]{\def\sSubHeadingName{#1}\begin{sSubHeading-r}}
  {\end{sSubHeading-r}} %
\newenvironment{demo*}[1]{\def\SubHeadingName{#1}\begin{SubHeading*}}
  {\end{SubHeading*}}%
\maketheorem{defini}{proclaim}{Definition}
\maketheorem{defnot}{proclaim}{Definitions and notation}
\maketheorem{question}{proclaim}{Question}
\maketheorem{terminology}{proclaim}{Terminology}
\maketheorem{crit}{proclaim}{Criterion}
\maketheorem{pitfall}{proclaim}{Pitfall}
\maketheorem{addition}{proclaim}{Addition}
\maketheorem{principle}{proclaim}{Principle} 
\maketheorem{condition}{proclaim}{Condition}
\maketheorem{exmp}{proclaim}{Example}
\maketheorem{hint}{proclaim}{Hint}
\maketheorem{exrc}{proclaim}{Exercise}
\maketheorem{prob}{proclaim}{Problem}
\maketheorem{ques}{proclaim}{Question}    
\maketheorem{alg}{proclaim}{Algorithm}
\maketheorem{remk}{proclaim}{Remark}          
\maketheorem{note}{proclaim}{Note}            
\maketheorem{summ}{proclaim}{Summary}         
\maketheorem{notationk}{proclaim}{Notation}   
\maketheorem{warning}{proclaim}{Warning}  
\maketheorem{defn-thm}{proclaim}{Definition--Theorem}  
\maketheorem{convention}{proclaim}{Convention}  
\maketheorem{hw}{proclaim}{Homework}
\maketheorem{hws}{proclaim}{\protect{${\mathbb\star}$}Homework}
\newtheorem*{ack}{Acknowledgment}
\newtheorem*{acks}{Acknowledgments}
\theoremstyle{bozont-number}
\theoremstyle{bozont-def}    
\maketheorem{defn}{proclaim}{Definition}
\maketheorem{subdefn}{equation}{Definition}
\theoremstyle{bozont-reverse}    
\maketheorem{corr}{proclaim}{Corollary} 
\maketheorem{lemr}{proclaim}{Lemma} 
\maketheorem{propr}{proclaim}{Proposition} 
\maketheorem{conjr}{proclaim}{Conjecture}
\theoremstyle{bozont-remark-reverse}
\newtheorem{SubHeading-r}[proclaim]{\SubHeadingName}
\newtheorem{sSubHeading-r}[equation]{\sSubHeadingName}
\newtheorem{SubHeadingr}[proclaim]{\SubHeadingName}
\newtheorem{sSubHeadingr}[equation]{\sSubHeadingName}
\theoremstyle{bozont-reverse-sc}
\newtheorem{proclaimr-special}[proclaim]{\specialthmname}
\newenvironment{proclaimspecialr}[1]%
{\def\specialthmname{#1}\begin{proclaimr-special}}%
{\end{proclaimr-special}}
\theoremstyle{bozont-remark-reverse}
\maketheorem{remr}{proclaim}{Remark}
\maketheorem{subremr}{equation}{Remark}
\maketheorem{notationr}{proclaim}{Notation} 
\maketheorem{assumer}{proclaim}{Assumptions} 
\maketheorem{obsr}{proclaim}{Observation} 
\maketheorem{exampler}{proclaim}{Example} 
\maketheorem{exr}{proclaim}{Exercise} 
\maketheorem{claimr}{proclaim}{Claim} 
\maketheorem{inclaimr}{equation}{Claim} 
\maketheorem{definir}{proclaim}{Definition}
\theoremstyle{bozont-def-newnum-reverse}    
\maketheorem{newnumr}{proclaim}{}
\theoremstyle{bozont-def-newnum-reverse-plain}
\maketheorem{newnumrp}{proclaim}{}
\theoremstyle{bozont-def-reverse}    
\maketheorem{defnr}{proclaim}{Definition}
\maketheorem{questionr}{proclaim}{Question}
\newtheorem{newnumspecial}[proclaim]{\specialnewnumname}
\newenvironment{newnum}[1]{\def\specialnewnumname{#1}\begin{newnumspecial}}{\end{newnumspecial}}
\theoremstyle{bozont-say}    
\newtheorem{say}[proclaim]{}
\theoremstyle{bozont-subsay}    
\newtheorem{subsay}[equation]{}
\theoremstyle{bozont-step}
\newtheorem{bstep}{Step}
\newcounter{thisthm} 
\newcounter{thissection} 
\newcommand{\ilabel}[1]{%
  \newcounter{#1}%
  \setcounter{thissection}{\value{section}}%
  \setcounter{thisthm}{\value{proclaim}}%
  \label{#1}}
\newcommand{\iref}[1]{%
  (\the\value{thissection}.\the\value{thisthm}.\ref{#1})}
\newcounter{lect}
\setcounter{lect}{1}
\newcommand\resetlect{\setcounter{lect}{1}\setcounter{page}{0}}
\newcommand\lecture{\newpage\centerline{\sfbf Lecture \arabic{lect}}
  \addtocounter{lect}{1}}
\newcommand\nnplecture{\hugeskip\centerline{\sfbf Lecture \arabic{lect}}
\addtocounter{lect}{1}}
\newcounter{topic}
\setcounter{topic}{1}
\newenvironment{topic}
{\noindent{\sc Topic 
\arabic{topic}:\ }}{\addtocounter{topic}{1}\par}
\counterwithin{equation}{proclaim}
\counterwithin{enumfulli}{equation}
\counterwithin{figure}{section} 
\newcommand\equinsect{\numberwithin{equation}{section}}
\newcommand\equinthm{\numberwithin{equation}{proclaim}}
\newcommand\figinthm{\numberwithin{figure}{proclaim}}
\newcommand\figinsect{\numberwithin{figure}{section}}
\newenvironment{sequation}{%
\setcounter{equation}{\value{thm}}
\numberwithin{equation}{section}%
\begin{equation}%
}{%
\end{equation}%
\numberwithin{equation}{proclaim}%
\addtocounter{proclaim}{1}%
}
\newcommand{\num}{\arabic{section}.\arabic{proclaim}}
\newenvironment{pf}{\smallskip \noindent {\sc Proof. }}{\qed\smallskip}
\newenvironment{enumerate-p}{
  \begin{enumerate}}
  {\setcounter{equation}{\value{enumi}}\end{enumerate}}
\newenvironment{enumerate-cont}{
  \begin{enumerate}
    {\setcounter{enumi}{\value{equation}}}}
  {\setcounter{equation}{\value{enumi}}
  \end{enumerate}}
\let\lenumi\labelenumi
\newcommand{\rmlabels}{\renewcommand{\labelenumi}{\rm \lenumi}}
\newcommand{\rmlabelsoff}{\renewcommand{\labelenumi}{\lenumi}}
\newenvironment{heading}{\begin{center} \sc}{\end{center}}
\newcommand\subheading[1]{\smallskip\noindent{{\bf #1.}\ }}
\newlength{\swidth}
\setlength{\swidth}{\textwidth}
\addtolength{\swidth}{-,5\parindent}
\newenvironment{narrow}{
  \medskip\noindent\hfill\begin{minipage}{\swidth}}
  {\end{minipage}\medskip}
\newcommand\nospace{\hskip-.45ex}
\newcommand{\sfbf}{\sffamily\bfseries}
\newcommand{\sfbfs}{\sffamily\bfseries\relsize{-.5}
}
\newcommand{\twidle}{\textasciitilde}
\makeatother

\newcounter{stepp}
\setcounter{stepp}{0}
\newcommand{\nextstep}[1]{%
  \addtocounter{stepp}{1}%
  \begin{bstep}%
    {#1}
  \end{bstep}%
  \noindent%
}
\newcommand{\resetsteps}{\setcounter{stepp}{0}}


\title[The relative Du~Bois complex]{The relative Du~Bois complex --- on a question of
  S.~Zucker}

\author{S\'andor J Kov\'acs}

\address{S\'andor Kov\'acs, University of Washington, Department of Mathematics,
  Seattle, WA 98195, USA}
\email{\href{mailto:skovacs@uw.edu }{skovacs@uw.edu }}
\urladdr{\href{http://www.math.washington.edu/~kovacs}{http://www.math.washington.edu/$\sim$kovacs}}

\author{Behrouz Taji} \address{Behrouz Taji, School of Mathematics and Statistics,
The University of New South Wales Sydney, NSW 2052 Australia}
\email{\href{mailto:b.taji@unsw.edu.au}{b.taji@unsw.edu.au}}
\urladdr{\href{https://web.maths.unsw.edu.au/~btaji//}
  {https://web.maths.unsw.edu.au/~btaji/}}

\date{\usdate\today} \thanks{S\'andor Kov\'acs was supported in part by NSF Grants
  DMS-1565352, DMS-1951376, and DMS-2100389.}
\maketitle
\setcounter{tocdepth}{1}
\numberwithin{equation}{section}

\centerline{\emph{To Vyacheslav Shokurov on the occasion of his $70+2^\text{nd}$
    birthday}}
\bigskip
  
\section{Introduction}
\noindent
Rational singularities form an extremley useful class and have provided a powerful
tool in the study of higher dimensional algebraic varieties. A prominent example of
rational singularities is provided by the main class of singularities used in the
minimal model program, that of \emph{klt singularities} \cite{Elkik81}.  The other
pillar of classification theory, besides the minimal model program, is moduli theory. In
fact, the minimal model program is a very useful tool for constructing moduli spaces.
In order to keep track of degenerations, or in other words to work with compact
moduli spaces, one needs to enlarge the class of singularities allowed from klt to
slc singularities.

Unfortunately, slc singularities are not always rational, so one needs to enlarge the
class of rational singularities as well. This enlargement is provided by the class of
Du~Bois singularities. In fact, the relationships between these two pairs of classes
of singularities are very similar. For instance, as mentioned above, klt
singularities are rational, and furthermore, for normal Gorenstein singularities
being rational and klt are equivalent. Similarly, slc singularities are Du~Bois, and
for normal Gorenstein singularities being Du~Bois is equivalent to slc.

Du~Bois singularities were introduced by Steenbrink
\cite{Steenbrink83}. Unfortunately, their definition is rather complicated and so the
reader is referred to \cite[\S 6]{SingBook} for details. For our purposes the
relevant detail is that the definition starts with the construction of the \DB
complex, which is an analogue of the de~Rham complex for not necessarily smooth
varieties.

From a cohomological point-of-view this complex behaves very much like the usual
de~Rham complex, including the existence and degeneration at $E_1$ of the
Hodge-to-de~Rham spectral sequence.  Just like the de~Rham complex, the \DB complex
may be used to acquire a Hodge structure and this actually agrees with Deligne's
Hodge structure on singular complex varieties.

In moduli theory we study families of varieties. An interesting question for smooth
families is how their Hodge structure is changing. In other words, for each smooth
family we obtain a variation of Hodge structure (of geometric origin). An effective
way to produce this is given by higher direct images of the relative de~Rham complex
of the family. As \DB complexes provide a way to obtain Deligne's Hodge structure on
singular varieties, it is a reasonable question to ask whether there is a way to
produce an object that would be analogous to the relative de~Rham complex of a smooth
family and which could be used to produce a variation of Hodge structure for a not
necessarily smooth family. This is probably too much to ask, but one may ask whether
there are some reasonable restrictions on the singularities of the fibers under which
this is possible.

In fact, this is a question that was posed by Steven Zucker at a JAMI conference at
Johns Hopkins University in 1996, organized by Vyacheslav Shokurov. Zucker's question
was motivated by a result of the first named author of the present paper in which an
analogue of the sheaf of relative $p$-forms was constructed. This construction has
been used successfully to prove various results for not necessarily smooth families
\cite{Kovacs96,Kovacs97c,Kovacs02,KT21,KT22}, however it fell short of a positive
answer to Zucker's question.

In this paper we make the first step towards filling that hole. More precisely we
prove the existence of a relative \DB complex for families parametrized by a smooth
curve. This is done in \autoref{sec:relative-db-complex}. We also establish some
expected properties of this complex in \autoref{sec:open-covers} and
\autoref{sec:functoriality}.  This work still leaves some questions open, most
notably a similar construction over arbitrary bases. We discuss this and other open
questions in \autoref{sec:open-questions}.

\section{The relative \DB complex} 
\label{sec:relative-db-complex}

In this section we will initially work in the abelian category of complexes of
abelian sheaves on a complex variety, instead of the corresponding derived category.

\begin{defini}
  Let $M$ be a complex manifold and let $\sA^{p,q}$ denote the sheaf of complex
  valued $C^\infty$-forms of type $(p,q)$ and let
  $\displaystyle\sA_M^{m}=\bigoplus_{p+q=m}\sA^{p,q}$.
  Next, let $X$ be a complex variety of dimension $n$ and $\hypres X$ a
  hyperresolution. Further let
  \[
    \ul\sA_X^m\leteq \bigoplus_{i=0}^n(\varepsilon_i)_*\sA_{X_i}^{m-i}=
    \bigoplus_{i=0}^n\bigoplus_{p=0}^{m-i}(\varepsilon_i)_*\sA_{X_i}^{p,m-i-p},
  \]
  with the filtration
  \[
    \left(F^p\cmx{\ul{\sA}}_X\right)^m\leteq
    \bigoplus_{i=0}^n\bigoplus_{r=p}^{m-i}(\varepsilon_i)_*\sA_{X_i}^{r,m-i-r}.
  \]
  The differential of $\cmx {\ul\sA}_X$ is a combination of the differentials of the
  complexes $\cmx\sA_{X_i}$ and the pull-back morphisms between the pieces of the
  hyperresolution $\varepsilon_\kdot$. For more details see \cite{Steenbrink85}.
  Note that the complex $\cmx {\ul\sA}_X$ is an incarnation of the filtered
  de~Rham-Du~Bois complex, i.e., its image in $D_{\filt}(X)$, the filtered derived
  category of $\sO_X$-modules, is isomorphic to $\cmx\Om_ X$. Further note that
  because $n=\dim X$, $F^p\cmx{\ul\sA}_X=0$ for $p>n$.
\end{defini}

Now, let $f:X\to C$ be a flat morphism from a complex variety $X$ to a smooth complex
curve $C$. Following the construction in \cite{Kovacs96}, we will construct a Koszul
filtration on $\cmx{\ul\sA}_X$. We define the map on simple tensors, and then extend
by linearity to the tensor product pre-sheaf and take the induced map on the
sheafification:
\[
  \xymatrix@R=.1ex{%
    \wedge:
    \cmx{\ul\sA}_X\otimes f^*\omega_C \ar[r] & \cmx{\ul\sA}_X [1] \\
    (\varepsilon_i)_*\eta_i\otimes \xi \ar@{|->}[r] & (\varepsilon_i)_*(\eta_i\wedge
    \varepsilon_i^*\xi) }
\]
This is a filtered map:
\[
  \xymatrix@R=.1ex{%
    \wedge_p\leteq \wedge\resto{F^p\cmx{\ul\sA}_X\otimes f^*\omega_C} :
    F^p\cmx{\ul\sA}_X\otimes f^*\omega_C \ar[r] & F^p(\cmx{\ul\sA}_X [1])=
    (F^{p+1}\cmx{\ul\sA}_X)[1] \\
    (\varepsilon_i)_*\eta_i\otimes \xi \ar@{|->}[r] & (\varepsilon_i)_*(\eta_i\wedge
    \varepsilon_i^*\xi). }
\]

Next, we will define a series of objects recursively. In order to simplify the
notation we will use the following shorthand: $F^p\leteq F^p\cmx{\ul\sA}_X$,
$\wedge'\leteq \wedge\otimes \id_{f^*\omega_C}$, and
$\wedge'_p\leteq \wedge_p\otimes \id_{f^*\omega_C}$. Because $\omega_C$ is a line
bundle, $\wedge\circ\wedge'=0$ and $\wedge_p\circ\wedge'_{p-1}=0$.

We define $E^p$ as follows.

For $p\geq n= \dim X$, let $E^p\leteq 0$,
$\sfw''_p\leteq 0\in \Hom_{C(X)}(F^p\otimes f^*\omega_C, E^p\otimes f^*\omega_C)$,
and $\sfw'_p\leteq 0\in \Hom_{C(X)}(E^p\otimes f^*\omega_C, F^{p+1}[1])$. Note that
then $\sfw'_p\circ \sfw''_p=0$, $\sfw''_p\circ\wedge'_{p-1}=0$, because $\sfw''_p=0$,
and $\wedge_p=0$, because $F^{p+1}=0$.

For $p<n$, assume that the following objects are defined for all $r>p$:
$E^r\in\obj (C(X))$,
$\sfw''_r\in \Hom_{C(X)}(F^r\otimes f^*\omega_C, E^r\otimes f^*\omega_C)$, and
$\sfw'_r\in \Hom_{C(X)}(E^r\otimes f^*\omega_C, F^{r+1}[1])$ such that
$\sfw'_r\circ \sfw''_r=\wedge_r$ and $\sfw''_r\circ\wedge'_{r-1}=0$.  Set
$\sfw_r\leteq \sfw''_r\otimes \id_{f^*\omega_C^{-1}}\in \Hom_{C(X)}(F^r, E^r)$, and
define
\[
  E^p\leteq \cone(\sfw_{p+1})\otimes f^*\omega_C^{-1},
\]
i.e.,
\[
  (E^p)^m= \left( \left(F^{p+1}\right)^{m+1}\oplus \left(E^{p+1}\right)^m\right)
  \otimes f^*\omega_C^{-1},
\]
with differential
\[
  d_{E^p}^m= \left(
    \begin{matrix}
      -d^{m+1}_{F^{p+1}} & 0 \\ & \\
      w^{m+1}_{p+1} & d^m_{E^{p+1}} 
    \end{matrix}
  \right) \otimes \id_{f^*\omega_C^{-1}}.
\]
Notice that according to this definition we have,
\begin{equation}
  \label{eq:4}
  E^{n-1}=F^n[1] \otimes f^*\omega_C^{-1}.
\end{equation}
Next, let
\[
  \xymatrix@C=4em{%
    \sfw''_p: F^p\otimes f^*\omega_C \ar[r] & E^p \otimes f^*\omega_C 
  }
\]
be defined by 
\[
  \xymatrix@C=4em{%
    F^p\otimes f^*\omega_C \ar[r]^-{(\wedge_p,0)} & F^{p+1}[1]\oplus E^{p+1}. }
\]
This is a morphism of complexes, because by assumption $\sfw''_{p+1}\circ\wedge_p=0$. It
follows directly from the definition that
\[
  \sfw''_p\circ\wedge'_{p-1}=\wedge_p\circ\wedge'_{p-1}=0.
\]
Finally,  let
\[
  \xymatrix@C=4em{%
    \sfw'_p: E^p\otimes f^*\omega_C \ar[r] & F^{p+1}[1]
  }
\]
be defined by
\[
  \xymatrix@C=4em{%
    F^{p+1}[1]\oplus E^{p+1} \ar[r]^-{(\id,0)} &  F^{p+1}[1].
  }
\]
Again, directly by the definition, we see that
\[
  \sfw'_p\circ \sfw''_p=\wedge_p.
\]
Iterating this construction we obtain that for each $p\in\bZ$ there exist \sess in
$C(X)$:
\begin{equation}
  \label{eq:1}
  \xymatrix{%
    0 \ar[r] & E^{p+1} \ar[r] & E^p\otimes f^*\omega_C \ar[r] & F^{p+1}[1] \ar[r] & 0. }
\end{equation}
The above constructions works even for negative $p$'s, although we will primarily be
interested in the case of $p\geq -1$. The following statement will be crucial later:
\begin{lem}\label{lem:Ep-is-a-complex}
  For each $p\in\bZ$, $E^{p+1}\subseteq E^p$ is a subcomplex.
\end{lem}

\begin{subrem}
  Notice that the construction of $\cmx E$ exhibits
  $E^{p+1}\subseteq E^p\otimes f^*\omega_C$ as a subcomplex (cf.~\autoref{eq:1}),
  so one might wonder if that conflicts with the statement of
  \autoref{lem:Ep-is-a-complex}. There is no conflict, because the two maps giving
  the two embeddings are independent of each other and they map to parts of the
  complexes $E^p$ and $E^p\otimes f^*\omega_C$ that do not correspond to each other
  when twisting with $f^*\omega_C$.
\end{subrem}

\begin{proof}
  We use descending induction on $p$. If $p\geq n-1$, then $E^{p+1}=0$ and hence the
  statement is trivially true.  Let $p$ be fixed and assume that
  $E^{r+1}\subseteq E^r$ is a subcomplex for each $r>p$.  Consider the following
  diagram:
  \[
    \xymatrix@C=4em{%
      \ar@{^(->}[d]       F^{p+2} \ar[r]^{\sfw_{p+2}} & E^{p+2} \ar@{^(->}[d] \\
      F^{p+1} \ar[r]^{\sfw_{p+1}} & E^{p+1}  }
  \]
  The left vertical arrow is an injection by definition, the right vertical arrow is
  an injection by the inductive hypothesis.  The diagram is commutative because the
  non-trivial part of $\sfw_r$ is $\wedge_r$ (for both $r=p+1$ and $r=p+2$) and
  $\wedge_{p+1}\resto {F^{p+2}}=\wedge_{p+2}$ by definition.
  It follows that then there is a commutative diagram of distinguished triangles:
  \[
    \xymatrix@C=4em{%
      \ar@{^(->}[d] F^{p+2} \ar[r]^{\sfw_{p+2}} & E^{p+2} \ar[r] \ar@{^(->}[d] &
      E^{p+1}\otimes f^*\omega_C \ar@{-->}[d]^{\delta_p'} \ar[r]^-{+1} &
      \\
      F^{p+1} \ar[r]^{\sfw_{p+1}} & E^{p+1} \ar[r] & E^{p}\otimes f^*\omega_C
      \ar[r]^-{+1} & }
  \]
  The broken arrow, denoted by $\delta_p'$, exists by the basic properties of mapping
  cones. Therefore turning the distinguished triangles around we obtain a morphism of
  \sess in $C(X)$:
  \begin{equation}
    \label{eq:10}
    \begin{aligned}
      \xymatrix@C=4em{%
        0 \ar[r] & E^{p+2} \ar[r] \ar@{^(->}[d] & E^{p+1}\otimes f^*\omega_C
        \ar[d]^{\delta_p'} \ar[r] & F^{p+2}[1] \ar@{^(->}[d] \ar[r] &
        0  \\
        0 \ar[r] & E^{p+1} \ar[r] & E^p\otimes f^*\omega_C \ar[r] & F^{p+1}[1] \ar[r]
        & 0 }
    \end{aligned}
  \end{equation}
  Then $\delta'_p$ is also injective by the five lemma. Twisting $\delta'_p$ with
  $\id_{f^*\omega_C^{-1}}$ exhibits $E^{p+1}$ as a subcomplex of $E^p$.
\end{proof}

\begin{defini}
  Let $\cmx\Om_{X/C}$ denote the object in $D_\filt(X)$ represented by $E^0$, the
  filtration given by $\cmx E$, i.e., $E^p\cmx\Om_{X/C}\leteq E^p$ for each
  $p\in\bN$, and call it the \emph{relative \DB complex} of the morphism $f:X\to C$.
  For later use, let us also set $E^{-1}\cmx\Om_{X/C}\leteq E^{-1}$.
\end{defini}

Recall that a similar object, $\Om_{X/C}^p$ was defined in \cite[1.3]{Kovacs96}. This
is the singular analogue of $\Omega_{X/C}^p$ for smooth morphisms. In the next
theorem we show the connection between these objects.

\begin{thm}\label{thm:assoc-graded}
  $\cmx\Om_{X/C}$ is a complex filtered by $E^{\kdot\geq 0}=\cmx E\cmx\Om_{X/C}$ such
  that its associated graded quotients satisfy the following in $D(X)$:
  \[
    Gr_E^p\cmx\Om_{X/C}[p] \qis \Om_{X/C}^p.
  \]
  Furthermore, $\cmx\Om_{X/C}$ is a bounded complex, i.e.,
  $\cmx\Om_{X/C}\in \obj D_\filt^b(X)$ and if $f$ is proper, then the cohmology
  sheaves of its associated graded quotients are coherent, i.e., in that case
  $\cmx\Om_{X/C}\in \obj D_{\filt,\coh}^b(X)$
\end{thm}

\begin{subrem}
  Note that the original construction/definition of the \DB complex of X uses
  simplicial or cubic hyperresolution and then it must be verified that the
  isomorphism class (in the appropriate derived category) of the obtained complex
  does not depend on the chosen hyperresolution.  In the construction of
  $\cmx\Om_{X/C}$ we did not make direct use of a hyperresolution of $X$, only a
  particular representative of the \DB complex. It is easy to see that the
  construction yields another complex in the same isomorphism class even if one uses
  a different representative. Hence, $\cmx\Om_{X/C}$, the isomorphism class (in the
  appropriate derived category) of the relative \DB complex of $X$ is also
  independent from the choice of hyperresolutions.
\end{subrem}

\begin{proof}
  Recall that by definition $E^0$ is a complex representing the isomorphism class of
  $\cmx\Om_{X/C}$ and its corresponding filtration is given by $E^p$ for
  $p\in\bN$. Similarly for $F^0$ representing $\cmx\Om_{X}$, where $E^p$ and $F^p$
  are as defined above.  Consider the following commutative diagram in $C(X)$:
  \begin{equation}
    \label{eq:3}
    \begin{aligned}
    \xymatrix{%
      & 0 \ar[d] & 0 \ar[d] & 0\ar[d]\\
      0 \ar[r] & E^{p+2}[p] \ar[r] \ar[d] & E^{p+1}[p]\otimes f^*\omega_C \ar[d]
      \ar[r] & F^{p+2}[p+1] \ar[d] \ar[r] &
      0  \\
      0 \ar[r] & E^{p+1}[p] \ar[d] \ar[r] & E^p[p]\otimes f^*\omega_C \ar[d] \ar[r] &
      F^{p+1}[p+1] \ar[d] \ar[r] &
      0 \\
      0 \ar[r] & Gr_E^{p+1}E^0
      [p] \ar[r] \ar[d] & Gr_E^pE^0
      [p]
      \otimes f^*\omega_C \ar[d] \ar[r] & Gr_F^{p+1}F^0
      [p+1] \ar[r] \ar[d] &
      0 \\
      & 0 & 0 & 0 }      
    \end{aligned}
  \end{equation}
  The columns are exact by the definition of the associated graded quotients, the
  first two rows are exact by \autoref{eq:1}, and then the last row is exact by the
  nine lemma.

  Recall that by definition $\Om_X^p=Gr_F^{p}\cmx\Om_{X}[p]$ and that for each $p$
  there exists a distinguished triangle in $D(X)$ (cf.\cite[(1.3.2)]{Kovacs96}),
  \begin{equation}
    \label{eq:2}
    \xymatrix{%
      \Om^p_{X/C}\otimes f^*\omega_C \ar[r] & \Om^{p+1}_X \ar[r] & \Om^{p+1}_{X/C}
      \ar[r]^-{+1} &. 
    }
  \end{equation}
  Because $F^{n+1}=0$, we have that $F^n\simeq Gr_F^{n}\cmx\Om_{X}$ and hence
  $\Om_X^n\qis F^n[n]$. Then by \autoref{eq:4}, \autoref{eq:2} (applied for $p=n-1$),
  and because $E^n=0$ and $\Om_{X/C}^n\qis 0$,we have the following
  quasi-isomorphisms in $D(X)$:
  \begin{equation}
    \label{eq:5}
    Gr_E^{n-1}\cmx\Om_{X/C}[n-1] \leteq E^{n-1}[n-1]\qis \Om_X^n\otimes f^*\omega_C^{-1}
    \qis \Om_{X/C}^{n-1}.
  \end{equation}
  Comparing \autoref{eq:3} and \autoref{eq:2} and using \autoref{eq:5} and descending
  induction on $p$ proves the first statement.

  Boundedness and coherence in the case when $f$ is proper follows from the
  construction and \cite[(1.3.5)]{Kovacs96}.
\end{proof}

\begin{cor}\label{cor:abs-to-rel}
  There exists a morphism $\cmx\Om_X\to \cmx\Om_{X/C}$ in $D_\filt^b(X)$, (and if $f$
  is proper, then in $D_{\filt, \coh}^b(X)$),
  which induces morphisms, for each $p$, that fit into the following distinguished
  triangles:
  \begin{equation}
    \label{eq:6}
    \xymatrix{%
      E^{p-1}\cmx\Om_{X/C}[-1]\otimes f^*\omega_C \ar[r] & F^p\cmx\Om_X \ar[r] &
      E^p\cmx\Om_{X/C}  \ar[r]^-{+1} & , 
    } 
  \end{equation}
  \begin{equation}
    \label{eq:9}
    \xymatrix{%
      E^{-1}\cmx\Om_{X/C}[-1]\otimes f^*\omega_C \ar[r] & \cmx\Om_X \ar[r] &
      \cmx\Om_{X/C}  \ar[r]^-{+1} & , 
    } \protect{\text{and}}
  \end{equation}
  \begin{equation}
    \label{eq:7}
    \xymatrix{%
      \Om^{p-1}_{X/C}\otimes f^*\omega_C \ar[r] & \Om^p_X \ar[r] &
      \Om^p_{X/C}  \ar[r]^-{+1} & ,
    }
  \end{equation}
\end{cor}

\begin{proof}
  \autoref{eq:6} follows directly from the construction of $E^p$, \autoref{eq:9} is
  the special case of \autoref{eq:6} with $p=-1$, and \autoref{eq:7} is
  \cite[(1.3.2)]{Kovacs96}. The only new information in this statement is that the
  second morphisms in each of these distinguished triangles are induced by a single
  filtered morphism $\cmx\Om_X\to \cmx\Om_{X/C}$. This follows from
  \autoref{thm:assoc-graded}.
\end{proof}

\noin
Next, we compare our construction to the existing relative de~Rham complex in the
case of a smooth family.

\begin{thm}\label{thm:smooth-morphisms}
  Let $f:X\to C$ be a smooth morphism from a (smooth) complex variety $X$ to a smooth
  complex curve $C$.  There exists a natural filtered isomorphism in $D_\filt^b(X)$,
  (and if $f$ is proper, then in $D_{\filt, \coh}^b(X)$),
  \[
    \cmx\Om_{X/C} \simeq \cmx\Omega_{X/C}.
  \]
\end{thm}

\begin{proof}
  If $f$ is smooth, then $\omega_{X/C}$ is a locally free sheaf and
  $\Om_{X/C}^p \simeq\Omega_{X/C}^p$ by \cite[(1.3.4)]{Kovacs96}. Using the diagram
  \autoref{eq:3} and descending induction on $p$ shows that the filtration $\cmx E$
  constructed earlier is simply the \emph{filtration b\^ ete} of the relative de~Rham
  complex $\cmx\Omega_{X/C}$, which in particular shows that desired statement.
\end{proof}

\begin{rem}
  Note that if $f$ is smooth, then by \autoref{thm:smooth-morphisms} the filtration
  $\cmx E$ becomes stationary downwards from $p=0$, i.e.,
  $E^0=E^{-1}=E^{-2}=\dots$. This implies that in this case 
  $E^{-1}\cmx\Om_{X/C}[-1]\simeq \cmx\Om_{X/C}[-1]$ and hence \autoref{eq:9}
  recovers the well-known \ses for smooth morphisms, cf.\cite[(X-11),
  p.250]{MR2393625}:
  \begin{equation*}
    \xymatrix{%
      0 \ar[r] &  \cmx\Omega_{X/C}[-1]\otimes f^*\omega_C \ar[r] & \cmx\Omega_X \ar[r] &
      \cmx\Omega_{X/C}  \ar[r] & 0.
    }
  \end{equation*}
\end{rem}

\section{Open covers}
\label{sec:open-covers}

\noin %
The construction of the relative \DB complex is invariant under restricting to an
open set:

\begin{prop}\label{prop:restrict-to-an-open}
  Let $f:X\to C$ be a flat morphism from a complex variety $X$ to a smooth complex
  curve $C$ and $U\subseteq X$ an open set. Then there exists a natural filtered
  isomorphism in $D_\filt^b(U)$, (and if $f\resto U:U\to C$ is proper, then in
  $D_{\filt, \coh}^b(U)$),
  \[
    \cmx\Om_{X/C}\resto U \simeq \cmx\Om_{U/C}.
  \]
\end{prop}

\begin{proof}
  This follows directly from the construction and \cite[3.10]{DuBois81}.
\end{proof}

\begin{cor}
  The construction of the relative \DB complex is invariant under an open base
  change. In other words, using the notation from \autoref{prop:restrict-to-an-open},
  further let $\jmath: V\into C$ be an open embedding and let
  $\imath: X_V=V\times_CX\into X$ the corresponding open embedding into $X$ and
  $f_V:X_V\to V$ the base change of $f$. Then there exists a natural filtered
  isomorphism in $D_\filt^b(X_V)$, (and if $f$ is proper, then in
  $D_{\filt, \coh}^b(X_V)$),
  \[
    \imath^*\cmx\Om_{X/C} \simeq \cmx\Om_{X_V/V}.
  \]
\end{cor}

\begin{proof}
  It is straightforward from \autoref{prop:restrict-to-an-open} that
  $\imath^*\cmx\Om_{X/C} \simeq \cmx\Om_{X_V/C}$. Next, note that as $f_V:X_V\to C$
  factors through $V$, in the construction of $\cmx\Om_{X_V/C}$, whenever
  $f_V^*\omega_C$ appears, it may be replaced by $f_V^*\omega_V$, because the two are
  actually naturally isomorphic. Hence $\cmx\Om_{X_V/C}\simeq \cmx\Om_{X_V/V}$, which
  proves the statement.
\end{proof}

Finally, we note that the relative \DB complex can be computed on an open cover in
the following sense: 

\begin{thm}
  Let $f:X\to C$ be a flat morphism from a complex variety $X$ to a smooth complex
  curve $C$ and let $\{\nu_i:U_i\into X\}$ be a finite open cover of $X$. Then there
  exists a natural filtered isomorphism in $D_\filt^b(X)$, (and if $f\resto{U_i}$ is
  proper for each $i$, then in $D_{\filt, \coh}^b(X)$),
  \[
    \cmx\Om_{X/C}\simeq \myR{\nu_\kdot}_*\cmx \Om_{U_\kdot/C}.
  \]
\end{thm}

\begin{rem}
  The reader not familiar with the notation used on the right hand side of this
  isomorphism, should consult \cite[p.46]{DuBois81} or \cite[1.1]{Kovacs96}.
\end{rem}

\begin{proof}
  This follows from the construction and \cite[(1.3.6)]{Kovacs96}.
\end{proof}

\section{Functoriality}
\label{sec:functoriality}

\noin%
In this section we explore the functorial properties of the construction of the
relative \DB complex.

\begin{notation}\label{not:funct}
  We will consider a commutative diagram,
  \[
    \xymatrix{%
      Y \ar[rr]^\phi \ar[rd]_-g &&  X \ar[ld]^f \\
      & C, }
  \]
  where $f$ and $g$ are flat morphisms, $X$ and $Y$ are complex varieties and $C$ is
  a smooth complex curve.
\end{notation}

\begin{thm}
  Under \autoref{not:funct} there exists a commutative diagram of natural filtered
  morphisms in $D_\filt^b(X)$, (and if $f$ is proper, then in
  $D_{\filt, \coh}^b(X)$),
  \[
    \xymatrix{%
      \cmx\Om_{X} \ar[r] \ar[d]  &  \myR\phi_*\cmx\Om_{Y} \ar[d]  \\
      \cmx\Om_{X/C}\ar[r] &  \myR\phi_*\cmx\Om_{Y/C},  }
  \]
  where the vertical morphisms are the ones obtained in \autoref{cor:abs-to-rel}.
\end{thm}

\begin{proof}
  Let $\cmx F_X$ and $\cmx F_Y$ denote the filtrations of $\cmx\Om_{X}$ and
  $\cmx\Om_{Y}$ and $\cmx E_X$ and $\cmx E_Y$ the filtrations of $\cmx\Om_{X/C}$ and
  $\cmx\Om_{Y/C}$ constructed above.  By \cite[(3.2.1)]{DuBois81} there exists a
  morphism $\gamma: \cmx F_X\to \myR\phi_*\cmx F_Y$ in $D_\filt^b(X)$. We choose and
  fix a complex in $C(X)$ representing the class of $\myR\phi_*\cmx F_Y$ such that
  this morphism is represented by a morphism of complexes. By abuse of notation we
  will use the same symbols for these complexes. Next we choose complexes
  representing $\myR\phi_*E_Y^p$ by descending induction for each $p\in\bN$, starting
  with the analogue of \autoref{eq:4} and following the same steps as on
  Page~\pageref{eq:4} until we get to the analogue of \autoref{eq:10}:
  \[
    \xymatrix@C=4em{%
      0 \ar[r] & \myR\phi_*E_Y^{p+1} \ar[r] & \myR\phi_*E_Y^p\otimes f^*\omega_C
      \ar[r] & \myR\phi_*F_Y^{p+1}[1] \ar[r] & 0 }
  \]
  We want to compare this \ses to the one on $X$, i.e., consider the following
  commutative diagram of \sess in $C(X)$ (we are using the above chosen complexes to
  represent the derived image objects). The morphisms will be explained below.
  \begin{equation}
    \label{eq:11}
    \begin{aligned}
      \xymatrix@C=4em{%
        0 \ar[r] & E_X^{p+1} \ar[r] \ar[d]^{\alpha_{p+1}} & E_X^{p}\otimes
        f^*\omega_C \ar@{-->}[d]^{\beta_p} \ar[r] & F_X^{p+1}[1] \ar[d]^{\gamma_p}
        \ar[r] &
        0  \\
        0 \ar[r] & \myR\phi_*E_Y^{p+1} \ar[r] & \myR\phi_*E_Y^p\otimes f^*\omega_C
        \ar[r] & \myR\phi_*F_Y^{p+1}[1] \ar[r] & 0 }
    \end{aligned}
  \end{equation}
  The morphism $\gamma_p$ (in $C(X)$!) is induced by the morphism $\gamma$ above.
  Using descending induction on $p$ we assume that a morphism $\alpha_{p+1}$ as
  indicated exists in $C(X)$.  Then the diagram \autoref{eq:11} shows that there
  exists a $\beta_p$ that makes it commutative and hence we may define
  $\alpha_p\leteq \beta_p\otimes\id_{f^*\omega_C^{-1}}$ for the next inductive step.

  Comparing this with the construction of these complexes shows that the induced
  morphisms on the associated graded quotients recover the similar compatible
  morphisms from \cite[(1.3.3)]{Kovacs96}.
\end{proof}


\section{Open questions}\label{sec:open-questions}

As it is obvious from the rest of this article, this construction is presently only
carried out when the base of a family is a smooth curve. The next step would be to
extend this construction to families over arbitrary bases.

\begin{demo}{\small\sf Problem}\label{prob-1}
  Is there a similar construction for families over arbitrary bases? More precisely,
  let $f:X\to S$ be a flat morphism from a complex variety $X$ to a smooth complex
  variety $S$. Is there an object $\cmx\Om_{X/S}$ with properties similar to the ones
  proved here in the case $\dim S=1$?
\end{demo}

\begin{subrem}
  It seems likely that the methods of \cite{Kovacs05a} and \cite{Kovacs97c} would
  allow this construction, but there are a few obstacles to be removed.
\end{subrem}

Perhaps the most important questions with respect to applications are with regard to
base change properties of this construction. We have established base change for open
embeddings.
However, the interesting question is arbitrary base change, especially restriction to
a point.

\begin{demo}{\small\sf Problem}\label{prob-2}
  Let $f:X\to C$ be a flat morphism from a complex variety $X$ to a smooth complex
  curve $C$ and let $t\in C$. Under what conditions does the following hold?
  \[
    \myL\jmath^*\cmx\Om_{X/C}\simeq \cmx\Om_{X_t},
  \]
  where $\jmath:\{t\}\into C$ is the embedding of $t$ in $C$.
\end{demo}

\begin{subrem}
  Of course, if $f$ is smooth, then $\cmx\Om_{X/C}\simeq\cmx\Omega_{X/C}$ consists of
  locally free sheaves and hence
  $\myL\jmath^*\cmx\Om_{X/C}\simeq \jmath^*\cmx\Omega_{X/C}\simeq \cmx\Omega_{X_t}$.
\end{subrem}

In case Problem~\ref{prob-1} has a positive solution, then one can ask the more general
question:

\begin{demo}{\small\sf Problem}
  let $f:X\to S$ be a flat morphism from a complex variety $X$ to a smooth complex
  variety $S$ and let $\tau: T\to S$ be a morphism from a smooth complex variety
  $T$. Under what conditions does the following hold?
  \[
    \myL\jmath^*\cmx\Om_{X/S}\simeq \cmx\Om_{X_T},
  \]
  where $X_T=X\times_ST$.
\end{demo}


\def\cprime{$'$} \def\cprime{$'$} \def\cprime{$'$} \def\cprime{$'$}
  \def\cprime{$'$} \def\polhk#1{\setbox0=\hbox{#1}{\ooalign{\hidewidth
  \lower1.5ex\hbox{`}\hidewidth\crcr\unhbox0}}} \def\cdprime{$''$}
  \def\cprime{$'$} \def\cprime{$'$} \def\cprime{$'$} \def\cprime{$'$}
  \def\cprime{$'$}
\providecommand{\bysame}{\leavevmode\hbox to3em{\hrulefill}\thinspace}
\providecommand{\MR}{\relax\ifhmode\unskip\space\fi MR}
\providecommand{\MRhref}[2]{%
  \href{http://www.ams.org/mathscinet-getitem?mr=#1}{#2}
}
\providecommand{\href}[2]{#2}

\end{document}